\documentclass[11pt]{amsart}

\usepackage{subfiles}

\usepackage{tabu}
\usepackage[margin=0.8in]{geometry}
\usepackage[dvipsnames]{xcolor}
\usepackage{graphicx}
\usepackage{amssymb}
\usepackage{mathrsfs}
\usepackage{amsthm}
\usepackage{amsmath}
\usepackage{stmaryrd}
\usepackage{tikz}
\usepackage{tikz-cd}
\usetikzlibrary{decorations.pathreplacing,decorations.markings}
\usepackage{accents}
\usepackage{upgreek}
\usepackage{enumerate}
\usepackage{bm}
\usepackage{mathtools}
\usepackage[all]{xy}
\usepackage{caption}

\usepackage{url}
\usepackage{float}
\usepackage{todonotes}
\usepackage{bbm}
\usepackage{colonequals}

\usepackage[full]{textcomp}
\usepackage[cal=cm]{mathalfa}
\usepackage{xparse}
\usepackage{comment}
\usepackage{cite}

\tikzset{
  commutative diagrams/.cd,
  arrow style=tikz,
  diagrams={>=stealth}
}

\theoremstyle{theorem}
\newtheorem{innercustomthm}{Theorem}
\newenvironment{customthm}[1]
  {\renewcommand\theinnercustomthm{#1}\innercustomthm}
  {\endinnercustomthm}

\makeatletter
\def\@tocline#1#2#3#4#5#6#7{\relax
  \ifnum #1>\c@tocdepth 
  \else
    \par \addpenalty\@secpenalty\addvspace{#2}%
    \begingroup \hyphenpenalty\@M
    \@ifempty{#4}{%
      \@tempdima\csname r@tocindent\number#1\endcsname\relax
    }{%
      \@tempdima#4\relax
    }%
    \parindent\z@ \leftskip#3\relax \advance\leftskip\@tempdima\relax
    \rightskip\@pnumwidth plus4em \parfillskip-\@pnumwidth
    #5\leavevmode\hskip-\@tempdima
      \ifcase #1
       \or\or \hskip 1em \or \hskip 2em \else \hskip 3em \fi%
      #6\nobreak\relax
    \dotfill\hbox to\@pnumwidth{\@tocpagenum{#7}}\par
    \nobreak
    \endgroup
  \fi}
\makeatother

\usetikzlibrary{calc}
\usetikzlibrary{fadings}
\usetikzlibrary{decorations.pathmorphing}
\usetikzlibrary{decorations.pathreplacing}

\newcounter{marginnote}
\DeclareMathAlphabet{\mathpzc}{OT1}{pzc}{m}{it}

\usepackage[backref=page]{hyperref}
\hypersetup{
  colorlinks   = true,          
  urlcolor     = blue,          
  linkcolor    = purple,          
  citecolor   = blue             
}

\makeatletter
\DeclareRobustCommand{\cev}[1]{%
  \mathpalette\do@cev{#1}%
}
\newcommand{\do@cev}[2]{%
  \fix@cev{#1}{+}%
  \reflectbox{$\m@th#1\vec{\reflectbox{$\fix@cev{#1}{-}\m@th#1#2\fix@cev{#1}{+}$}}$}%
  \fix@cev{#1}{-}%
}
\newcommand{\fix@cev}[2]{%
  \ifx#1\displaystyle
    \mkern#23mu
  \else
    \ifx#1\textstyle
      \mkern#23mu
    \else
      \ifx#1\scriptstyle
        \mkern#22mu
      \else
        \mkern#22mu
      \fi
    \fi
  \fi
}

\theoremstyle{theorem}
\newtheorem{theorem}{Theorem}[section]

\newtheorem{corollary}[theorem]{Corollary}
\newtheorem{lemma}[theorem]{Lemma}
\newtheorem{proposition}[theorem]{Proposition}
\newtheorem{specification}[theorem]{Specification}

\theoremstyle{definition}
\newtheorem{remark}[theorem]{Remark}

\newtheorem*{runningexample*}{Running example}

\newtheorem*{aside*}{Aside}

\newtheorem{definition}[theorem]{Definition}

\newtheorem{proposition-definition}[theorem]{Proposition-Definition}

\newcommand{\RR}{\mathbb{R}}

\newcommand{\olM}{\overline{M}}
\newcommand{\Gm}{\mathbb{G}_{\operatorname{m}}}

\newcommand{\ol}[1]{\overline{#1}}

\newcommand{\bcd}{\begin{center}\begin{tikzcd}}
\newcommand{\ecd}{\end{tikzcd}\end{center}}

\newcommand{\Aaff}{\mathbb{A}}

\newcommand{\PP}{\mathbb{P}}
\newcommand{\OO}{\mathcal{O}}
\newcommand{\N}{\mathbb{N}}
\newcommand{\Z}{\mathbb{Z}}
\newcommand{\Q}{\mathbb{Q}}

\newcommand{\R}{\mathbb{R}}

\newcommand{\Speck}{\operatorname{Spec}\kfield}
\newcommand{\kfield}{\Bbbk}

\newcommand{\Bcal}{\mathcal{B}}

\newcommand{\Acal}{\mathcal{A}}
\newcommand{\Dcal}{\mathcal{D}}

\newcommand{\Ccal}{\mathcal{C}}
\newcommand{\Mfrak}{\mathfrak{M}}

\newcommand{\Mlog}{\overline{M}_{0,\upalpha}(\PP^r|H)}

\newcommand{\Mbar}{\ol{M}}

\newcommand{\ClQ}{\mathsf{Cl}_{\mathbb{Q}}}

\newcommand{\PicQ}{\mathsf{Pic}_{\mathbb{Q}}}

\newcommand{\Trop}{\Sigma}

\newcommand{\Mbarn}{\ol{{M}}_{0,n}}

\definecolor{FigureRed}{HTML}{EF4733}
\definecolor{FigureBlue}{HTML}{2473F6}

\begin{document}
 
\title{Divisors and curves on logarithmic mapping spaces}
\author{Patrick Kennedy-Hunt, Navid Nabijou, Qaasim Shafi, Wanlong Zheng}

\begin{abstract}
We determine the rational class and Picard groups of the moduli space of stable logarithmic maps in genus zero, with target projective space relative a hyperplane. For the class group we exhibit an explicit basis consisting of boundary divisors. For the Picard group we exhibit a spanning set indexed by piecewise-linear functions on the tropicalisation. In both cases a complete set of boundary relations is obtained by pulling back the WDVV relations from the space of stable curves. Our proofs hinge on a controlled technique for manufacturing test curves in logarithmic mapping spaces, opening up the topology of these spaces to further study.
\end{abstract}

\maketitle
\tableofcontents

\section*{Introduction}

\noindent Fix integers $r \geqslant 2,n \geqslant 3, d \geqslant 1$ and a non-negative ordered partition $\upalpha =(\upalpha_1,\ldots,\upalpha_n) \vdash d$. This discrete data determines a moduli space
\[ \Mlog \]
of stable logarithmic maps, where $H \subseteq \PP^r$ is a hyperplane \cite{GrossSiebertLog,ChenLog,AbramovichChenLog}. It is a logarithmically smooth Deligne--Mumford stack, with dense interior parametrising degree $d$ maps $f \colon \PP^1 \to \PP^r$ satisfying
\[ f^\star H = \Sigma_{i=1}^n \upalpha_i x_i\]
where $x_1,\ldots,x_n \in \PP^1$ are marked points (note that $\upalpha$ determines both $n$ and $d$). We investigate the divisor theory of this space.

\subsection{Faithful tropicalisation} As a preliminary result, we identify the boundary complex of $\Mlog$ with a moduli space of stable tropical maps.

\begin{customthm}{X}[Theorem~\ref{thm: faithful tropicalisation}] \label{thm: faithful tropicalisation introduction} There is a natural isomorphism of cone complexes
\[ \Sigma\Mlog = \Sigma_{0,\upalpha} \]
where $\Sigma \Mlog$ is the tropicalisation of the moduli space of stable logarithmic maps, and $\Sigma_{0,\upalpha}$ is the moduli space of stable tropical maps to $\R_{\geqslant 0}$ (see Section~\ref{sec: tropical moduli}).	
\end{customthm}
In particular, the number $N(\upalpha)$ of irreducible boundary divisors in $\Mlog$ is equal to the number of one-dimensional combinatorial types of stable tropical map to $\R_{\geqslant 0}$ with tangency profile $\upalpha$. Calculating $N(\upalpha)$ for specific values of $\upalpha$ is a combinatorial exercise; see Section~\ref{sec: enumerating boundary strata} for a discussion in the case $\upalpha=(d,0,\ldots,0)$.

\subsection{Class group} We begin with Weil divisors. The key geometric input is a method for manufacturing test curves in the moduli space with strongly controlled intersection against the boundary.

\begin{customthm}{Y}[Theorem~\ref{thm: main}] \label{thm: main introduction} The rational class group of $\Mlog$ has an explicit basis consisting of boundary divisors. It has dimension
\[ N(\upalpha) - {n-1 \choose 2} + 1. \]
In particular, the dimension of the rational class group does not depend on $r$. Moreover, all relations between boundary divisors are pulled back from the WDVV relations on $\Mbarn$.
\end{customthm}
In the above formula, $N(\upalpha)$ counts the number of irreducible boundary divisors (see Theorem~\ref{thm: faithful tropicalisation introduction}), while the remaining terms count the number of independent linear relations between these divisors.

\subsection{Picard group} \label{sec: class group not Picard} The space $\Mlog$ is logarithmically smooth, i.e. it has toric singularities. It is almost never smooth, not even as an orbifold; consider e.g. the following combinatorial type of stable tropical map, with $\upalpha=(2,0,0)$
\[
\begin{tikzpicture}[scale=0.95]
\draw[fill=black] (-2,0.75) circle[radius=2pt];
\draw[blue] (-2,0.75) node[left]{\scriptsize$1$};

\draw (-2,0.75) -- (-0.5,0.75);
\draw (-1.25,0.75) node[above]{\scriptsize$e_1$};

\draw[fill=black] (-0.5,0.75) circle[radius=2pt];
\draw[blue] (-0.5,0.75) node[below]{\scriptsize$0$};

\draw[->] (-0.5,0.75) -- (-0.5,1.25);
\draw[red] (-0.5,1.2) node[above]{\scriptsize$0$};

\draw (-0.5,0.75) -- (1,0);
\draw (0.35,0.3) node[above]{\scriptsize$e_2$};

\draw[fill=black] (-2,-0.75) circle[radius=2pt];
\draw[blue] (-2,-0.75) node[left]{\scriptsize$1$};

\draw (-2,-0.75) -- (-1,-0.75);
\draw (-1.5,-0.75) node[above]{\scriptsize$e_3$};

\draw[fill=black] (-1,-0.75) circle[radius=2pt];
\draw[blue] (-1,-0.75) node[below]{\scriptsize$0$};
\draw[->] (-1,-0.75) -- (-1,-0.25);
\draw[red] (-1,-0.3) node[above]{\scriptsize$0$};

\draw (-1,-0.75) -- (1,0);
\draw (-0.1,-0.4) node[above]{\scriptsize$e_4$};

\draw[fill=black] (1,0) circle[radius=2pt];
\draw[blue] (1,0) node[below]{\scriptsize$0$};

\draw[->] (1,0) -- (1.75,0);
\draw[red] (1.7,0) node[right]{\scriptsize$2$};

\draw[fill=blue,blue] (-2,-2.25) circle[radius=2pt];
\draw[blue,->] (-2,-2.25) -- (1.75,-2.25);
\draw[blue] (1.75,-2.25) node[right]{\scriptsize$H$};

\draw[blue,->] (0,-1) -- (0,-1.75);
\draw[blue] (0,-1.3) node[right]{\scriptsize$\mathsf{f}$};

\end{tikzpicture}
\]
The edge lengths satisfy $e_1+e_2=e_3+e_4$ and so the local singularity type of $\Mlog$ is given by the quadric cone
\[ V(z_1 z_2 - z_3 z_4) \subseteq \Aaff^{\! 4}_{\underline{z}}.\]
It follows that $\Mlog$ is not even $\Q$-factorial, and most boundary components are not even $\Q$-Cartier. This contrasts sharply with the space $\ol{M}_{0,n}(\PP^r,d)$ of ordinary stable maps.


Consequently, the class group and the Picard group of $\Mlog$ differ. We leverage our complete understanding of the former to obtain a combinatorial description of the latter.

\begin{customthm}{Z}[Theorem~\ref{thm: Picard}] \label{thm: Picard introduction} The rational Picard group of $\Mlog$ is generated by divisors corresponding to integral piecewise-linear functions on $\Sigma_{0,\upalpha}$. A complete set of relations between these divisors is obtained by pulling back the WDVV relations from $\Mbarn$.	
\end{customthm}

\subsection{Strategy} Theorems~\ref{thm: main introduction} and \ref{thm: Picard introduction} are analogous to \cite[Theorem~2]{PandPic} (see also \cite{Oprea} for a more general calculation). The basic strategy is similar: to probe divisors in the moduli space by constructing appropriate test curves
\[ \PP^1 \to \Mlog \]
with controlled intersection against the boundary. The details, however, are different. We do not start with an arbitrary test curve and then modify it. Instead, we directly manufacture a roster of test curves which is sufficiently rich and sufficiently controlled to establish the required linear independences. This alternative approach leads directly to an explicit basis for the class group. The test curve construction is considerably more intricate than that of \cite{PandPic} due to the combinatorial complexity of boundary divisors in the logarithmic mapping space, and the strong control we maintain over specific boundary intersections.

The proof proceeds as follows. We first use excision to show that the class group is generated by boundary divisors (Proposition~\ref{prop: generated by boundary strata}).\footnote{In contrast, boundary divisors are insufficient to generate the class group of $\overline{M}_{0,n}(\PP^r,d)$, see Remark~\ref{rmk: absolute stable maps not generated by boundary divisors}.} We then partition these into \textbf{aliens}, \textbf{airbornes} and \textbf{terrestrials} (Definition~\ref{def: alien and terrestrial}). An alien divisor has generic point parametrising a curve with two irreducible components, one of which contains two markings and has degree $0$, the other of which contains all the other markings and has degree $d$. The airborne divisor has generic point parametrising a smooth curve mapped inside $H$. Boundary divisors which are neither alien nor airborne are referred to as terrestrial.

A consequence of our test curve construction (Section~\ref{sec: test curves}) is that for every terrestrial divisor $D$, there exists a test curve which intersects $D$ and does not intersect any other terrestrial divisor. It follows that the terrestrial divisors form a linearly independent subset of the class group. Routine surgery on test curves then shows that the same is true for the set of terrestrial and airborne divisors (Proposition~\ref{prop: airborne and terrestrial linearly independent set}).

Quotienting, we are left with a space spanned (but not based) by the alien divisors. We show that this quotient has dimension $n$, by establishing upper and lower bounds (Propositon~\ref{prop: dim of Q}). The upper bound is obtained by studying pullbacks of relations from the space of stable curves. The lower bound is obtained by exhibiting a linearly independent set \eqref{eqn: linearly independent set in Q} of alien divisors of size $n$. As before, linear independence is demonstrated by intersecting with suitable test curves.

This shows that the class group has a basis consisting of the airborne divisor, the terrestrial divisors, and the alien divisors listed in \eqref{eqn: linearly independent set in Q}. As a corollary we conclude that all relations amongst boundary divisors are pulled back from the WDVV relations on $\Mbarn$. This establishes Theorem~\ref{thm: main introduction}.

The Picard group embeds in the class group as the set of locally principal divisors. Since the class group is generated by boundary divisors, the Picard group is generated by boundary Cartier divisors, and these are indexed by piecewise-linear functions on the tropicalisation. Finally, all relations in the Picard group come from relations in the class group and hence, by Theorem~\ref{thm: main introduction}, from relations on $\Mbarn$. This establishes Theorem~\ref{thm: Picard introduction}.

It is perhaps surprising that there exist enough test curves to establish all linear independences between divisors. Our construction of test curves shows that
\begin{equation} \dim A_1(\Mlog)_{\Q} \geqslant \dim A_{n-1}(\Mlog)_{\Q}.\end{equation}
There is no a priori reason why this should hold, since Poincar\'e duality can fail for varieties with toric singularities, see \cite[Example~4.2]{KatzPayne} and \cite[Section~1]{TotaroToric}.

\subsection{Outlook} We expect the test curve construction of Section~\ref{sec: test curves} and the linear independence arguments of Section~\ref{sec: relations} to open up the study of the topology of logarithmic mapping spaces.\medskip

\noindent \textbf{Problem 0.1.} Investigate piecewise-linear functions on $\Sigma_{0,\upalpha}$.\label{problem Picard}

Theorem~\ref{thm: Picard introduction} is less explicit than Theorem~\ref{thm: main introduction}, because on $\Sigma_{0,\upalpha}$ it is harder to enumerate piecewise-linear functions than it is to enumerate rays. A more conceptual description of these piecewise-linear functions is desirable, e.g. in terms of ``tautological'' functions arising from the modular interpretation of $\Sigma_{0,\upalpha}$.\medskip

\noindent \textbf{Problem 0.2.} Investigate divisors on spaces of prestable logarithmic maps to the Artin fan $[\Aaff^{\!1}/\Gm]$.

A minor modification of our test curve construction (simply dropping Section~\ref{sec: constructing f1 to fm}) also produces test curves in such spaces. Boundary relations on $\Mfrak_{0,n}$ will likely play a role \cite{BaeSchmitt1,BaeSchmitt2}. New ideas are required to handle curve components with negative degrees. This will also open up the divisor theory of spaces of logarithmic maps to general smooth pairs.\medskip

\noindent \textbf{Problem 0.3.} Investigate divisors on $\overline{M}_{0,\upalpha}(\PP^r|H_0+\ldots+H_k)$ for $H_0+\ldots+H_k$ a subset of the toric boundary.

Such spaces are also logarithmically smooth. The test curve construction must be modified, for which a systematic understanding of the possible shapes of boundary divisors is required, analogous to Proposition~\ref{prop: comb types of boundary divisors}. The case of full toric boundary, at least, is well-understood \cite{RanganathanSkeletons1}.\medskip

\noindent \textbf{Problem 0.4.} Investigate higher codimension cycles on logarithmic mapping spaces.

This requires a more systematic understanding of higher codimension boundary strata, and a general method for manufacturing higher dimension test varieties (see \cite{TarascaBrill,TarascaDoubleTotal}). First steps towards the higher codimension topology of logarithmic mapping spaces are taken in \cite{KannanP1}.\medskip

\noindent \textbf{Problem 0.5.} Investigate cycles with $\Z$ coefficients.

It is still possible to intersect our test curves with boundary divisors in this context. However, care is required around intersection multiplicities and stabiliser groups.

\subsection{Outline} In Section~\ref{sec: generators} we establish a faithful tropicalisation result (Theorem~\ref{thm: faithful tropicalisation}) which allows us to combinatorially enumerate boundary divisors (Corollary~\ref{lem: boundary indexed by types}). We then use excision to show that the class group is generated by these boundary divisors (Proposition~\ref{prop: generated by boundary strata}).

In Section~\ref{sec: test curves} we outline a general method for manufacturing test curves in the moduli space. Our test curves are morphisms $\PP^1 \to \Mlog$ which intersect the boundary at finitely many points away from strata of codimension $\geqslant 2$. We establish strong control over which boundary strata a test curve can intersect.

In Section~\ref{sec: relations} we prove the main results. We furnish an explicit basis for the rational class group, consisting of boundary divisors (Theorem~\ref{thm: main}). The test curve construction is crucial for proving that various collections of boundary divisors are linearly independent. As a corollary, we conclude that all relations amongst boundary divisors are pulled back from the WDVV relations on $\Mbarn$. Finally in Section~\ref{sec: Picard} we leverage our control over the class group to give a combinatorial description of the Picard group (Theorem~\ref{thm: Picard}).

In Appendix~\ref{sec: divisors on Mbarn} we record for posterity a complete linearly independent set of relations between boundary divisors in $\Mbarn$.

\subsection{Acknowledgements} We thank Dhruv~Ranganathan for initially suggesting this problem, and for several helpful conversations. We thank Mark~Gross for useful discussions, and Yu~Wang for collaborations in the early stages of the project. We thank the referee for several useful suggestions.

P.K.-H. is supported by an EPSRC Studentship, reference 2434344. N.N. is supported by the Herchel Smith Fund. Q.S. is supported by EPSRC Centre for Doctoral Training in Geometry and Number Theory at the Interface, grant number EP/L015234/1. W.Z. is supported by Cambridge International Trust and DPMMS at the University of Cambridge.

\section{Generators} \label{sec: generators}

\noindent For background on stable logarithmic maps, stable tropical maps, and combinatorial types thereof, we refer to \cite[Section~2]{ACGSDecomposition}.

\subsection{Tropical moduli} \label{sec: tropical moduli} A \textbf{combinatorial type of stable tropical map} to $\R_{\geqslant 0}$ with tangency profile $\upalpha$ consists of the following data:
\begin{itemize}
\item \textbf{Source graph.} A finite tree $\Gamma$ consisting of vertices $V(\Gamma)$, finite edges $E(\Gamma)$, and legs $L(\Gamma)$ equipped with a bijection
\[ L(\Gamma) \cong \{1,\ldots,n\}.\]
Each vertex $v \in V(\Gamma)$ is equipped with a degree label $d_v \in \N$ such that $\Sigma_{v \in V(\Gamma)} d_v=d$.\smallskip
\item \textbf{Image cones.} Faces of $\R_{\geqslant 0}$ associated to every vertex and edge of $\Gamma$
\[ v \rightsquigarrow \upsigma_v, \quad e \rightsquigarrow \upsigma_e, \]
such that if $v \leqslant e$ then $\upsigma_v \leqslant \upsigma_e$.\smallskip
\item \textbf{Slopes.} For each oriented edge $\vec{e} \in \vec{E}(\Gamma)$ a slope $m_{\vec{e}} \in \Z$ satisfying $m_{\vec{e}}=-m_{\cev{e}}$. At each vertex $v \in V(\Gamma)$ these slopes must satisfy the balancing condition
\[ d_v = \sum_{v \leqslant e} m_{\vec{e}} + \sum_i \upalpha_i \]
where each edge $e$ is oriented away from $v$, and the second sum is over legs supported at $v$.
\end{itemize}

Each vertex is required to be stable, meaning that either $d_v > 0$ or $v$ is at least trivalent. To each combinatorial type we associate a strictly convex rational polyhedral cone constituting the associated tropical moduli. This is embedded in an orthant coordinatised by the edge lengths $\ell_e$ and the vertex positions $\mathsf{f}(v)$ and cut out by continuity equations
\[ \mathsf{f}(v_1) + m_{\vec{e}}\, \ell_e = \mathsf{f}(v_2) \]
indexed by oriented edges $\vec{e}$ from $v_1$ to $v_2$. The \textbf{dimension} of a combinatorial type is the dimension of the corresponding cone. Specialisations of combinatorial types induce face inclusions, and the colimit of the resulting diagram is a cone complex which we denote
 \[ \Sigma_{0,\upalpha} \]
 and refer to as the \textbf{moduli space of stable tropical maps.}

\subsection{Boundary and tropicalisation} Since the pair $(\PP^r|H)$ is convex, the space $\Mlog$ is logarithmically smooth over the trivial logarithmic point, i.e. is a toroidal embedding \cite{KKMSD, AbramovichKaru}. It has dense interior
\[ M_{0,\upalpha}(\PP^r|H) \subseteq \Mlog\]
 parametrising maps from a smooth source curve that do not factor through the hyperplane $H$. The \textbf{boundary} is the complement
\[ \partial \Mlog \colonequals \Mlog \setminus M_{0,\upalpha}(\PP^r|H). \]
It is a union of irreducible hypersurfaces. The tropicalisation 
\[ \Trop \Mlog\]
is a cone complex whose cones are in bijective, inclusion-reversing correspondence with boundary strata. Tropicalisation of families of stable logarithmic maps \cite[Section~2.5]{ACGSDecomposition} produces a natural morphism of cone complexes
\[ \varphi \colon \Trop \Mlog \to \Sigma_{0,\upalpha}.\]
The following is a case of faithful tropicalisation of moduli spaces \cite{CaporasoGonality,ACP,UlirschWeightedStable,AscherMolcho,CMRAdmissible,CHMRWeighted,GrossCorrespondence,RanganathanSkeletons1,BBCMMW,CavalieriChanUlirschWise,MR20, LenUlirsch,OdakaOshima,MolchoWise,TropUnivJac,BCK,KHQuot,NabijouConfiguration}.

\begin{theorem}[Theorem~\ref{thm: faithful tropicalisation introduction}] \label{thm: faithful tropicalisation} $\varphi$ is an isomorphism of cone complexes.	
\end{theorem}

\begin{proof}
The tropical interpretation of basicness for stable logarithmic maps \cite[Section~1.4]{GrossSiebertLog} ensures that $\varphi$ maps every source cone isomorphically onto a target cone. Hence it suffices to show that every target cone is the image of a unique source cone.

A target cone $\uptau \leqslant \Sigma_{0,\upalpha}$ corresponds to a combinatorial type of stable tropical map. Let
\[ M_\uptau(\PP^r|H) \subseteq \Mlog \]
denote the locally-closed locus of logarithmic maps which tropicalise to this combinatorial type. A standard lifting argument for smooth pairs (see e.g. \cite[Lemma~3.1]{BNR2}) shows that $M_\uptau(\PP^r|H)$ is nonempty. This ensures that there exists at least one cone in $\Sigma \Mlog$ whose image is $\uptau$. To show that this cone is unique, we must show that $M_\uptau(\PP^r|H)$ is irreducible.

Consider the tower of forgetful morphisms
\begin{equation} \label{eqn: tower of forgetful morphisms} \Mlog \to \Mfrak_{0,\upalpha}(\Acal|\Dcal) \to \Mfrak_{0,n}^{\log} \to \Mfrak_{0,n}.\end{equation}
where $\Mfrak_{0,\upalpha}(\Acal|\Dcal)$ is the space of prestable logarithmic maps to the Artin fan and $\Mfrak_{0,n}^{\log}$ is the space of logarithmic curves. For background, see \cite[Section~3]{AbramovichWiseBirational}, \cite[Section~3.5]{AbramovichMarcusWiseComparison}, and \cite[Appendix~A]{GrossSiebertLog}.

The proof strategy is to pass up the tower \eqref{eqn: tower of forgetful morphisms}, constructing the locally-closed stratum $M_\uptau(\PP^r|H)$ inductively. At each step we deduce irreducibility of the given stratum by examining the fibres over the previous stratum. Starting at the bottom, the combinatorial type of stable tropical map indexing $\uptau$ contains the data of the marked dual graph $\Gamma$ of the source curve. This defines a locally-closed stratum
\[ \Mfrak_\Gamma \subseteq \Mfrak_{0,n}.\]
Note that $\Mfrak_\Gamma$ is irreducible. It fits into the following fibre square (see e.g. \cite[Section~1.4]{CN22})
\[
\begin{tikzcd}
\Mfrak_\Gamma \ar[r,hook] \ar[d] \ar[rd,phantom,"\square"] & \Mfrak_{0,n} \ar[d] \\
\Bcal T_{\upsigma_\Gamma} \ar[r,hook] & \Acal_{\Mfrak_{0,n}}
\end{tikzcd}
\]
where $\upsigma_\Gamma \leqslant \Sigma \Mfrak_{0,n}$ is the smooth cone coordinatised by the edge lengths of $\Gamma$. There is a natural morphism of cones $\uptau \to \upsigma_\Gamma$ and a locally-closed stratum 
\[ \Mfrak_\uptau^{\log} \subseteq \Mfrak_{0,n}^{\log} \]
parametrising logarithmic curves whose tropicalisation is the pullback of the universal tropical curve along $\uptau \to \upsigma_\Gamma$. This fits into the following fibre square (see e.g. \cite[Corollary~5.25]{OlssonLogStacks})
\[
\begin{tikzcd}
\Mfrak_\uptau^{\log} \ar[r] \ar[d] \ar[rd,phantom,"\square"] & \Mfrak_\Gamma \ar[d] \\
\Bcal T_\uptau \ar[r] & \Bcal T_{\upsigma_\Gamma}.	
\end{tikzcd}
\]
The lattice morphism $N_\uptau \to N_{\upsigma_\Gamma}$ factors through its image
\[ N_\uptau \twoheadrightarrow N_{\widetilde\uptau} \hookrightarrow N_{\upsigma_\Gamma} \]
and the same holds for the map of classifying stacks
\[ \Bcal T_\uptau \to \Bcal T_{\widetilde\uptau} \to \Bcal T_{\upsigma_\Gamma}.\]
A direct argument (see e.g. \cite[Lemma~4.5]{CN22}) shows that the first morphism is a gerbe for
\[ \ker(N_{\uptau} \to N_{\widetilde\uptau})\otimes \Gm\]
while the second morphism is a principal bundle for
\[ \operatorname{coker}(N_{\widetilde\uptau} \to N_{\upsigma_\Gamma})\otimes \Gm.\]
Both these groups are algebraic tori; in the latter case, because the inclusion $N_{\widetilde\uptau} \hookrightarrow N_{\upsigma_\Gamma}$ is saturated. 
In particular, they are irreducible. It follows that the generic fibre of
\[ \Mfrak_\uptau^{\log} \to \Mfrak_\Gamma\] 
is irreducible, and hence $\Mfrak_\uptau^{\log}$ is irreducible. Next, consider the locally-closed stratum
\[ \Mfrak_\uptau(\Acal|\Dcal) \subseteq \Mfrak_{0,\upalpha}(\Acal|\Dcal) \]
parametrising prestable logarithmic maps to the universal target whose combinatorial type is that indexing $\uptau$. We claim that the forgetful morphism
\begin{equation} \label{eqn: map on strata, Artin fan to log curves} \Mfrak_\uptau(\Acal|\Dcal) \to \Mfrak_\uptau^{\log} \end{equation}
is bijective on geometric points. Indeed given a logarithmic curve $\Ccal$, a morphism $\Ccal \to (\Acal|\Dcal)$ is equivalent to a piecewise-linear function on $\Sigma \Ccal$ (see e.g. \cite[Proposition~2.10]{ACGSDecomposition}). However this piecewise-linear function is already determined by the choice of combinatorial type, and hence there is a unique lift. We conclude that $\Mfrak_\uptau(\Acal|\Dcal)$ is irreducible.

Finally we describe the fibres of 
\begin{equation} \label{eqn: map on strata, geometric target to universal target} M_\uptau(\PP^r|H) \to \Mfrak_\uptau(\Acal|\Dcal).\end{equation}
Choose coordinates on $\PP^r$ so that $H=H_0$ is the first coordinate hyperplane. A geometric point of $\Mfrak_\uptau(\Acal|\Dcal)$ consists of a marked logarithmic curve $\Ccal$ together with a piecewise-linear function on $\Sigma \Ccal$ with associated line bundle-section pair $(L,f_0)$. A lift to $M_\uptau(\PP^r|H)$ consists of the data of additional sections
\[ f_1,\ldots,f_r \in H^0(C,L)\]
such that the tuple $[f_0,f_1,\ldots,f_r]$ has no basepoints. The space of such sections is a Zariski open subset of affine space, hence is irreducible. We conclude that $M_\uptau(\PP^r|H)$ is irreducible.
\end{proof}

\begin{remark}
The specific geometry of $(\PP^r|H)$ is invoked only in the final step of the preceding proof, where we utilise our understanding of the fibres of \eqref{eqn: map on strata, geometric target to universal target}. For a general smooth pair such fibres are difficult to describe: they may be empty, or may have multiple components. For examples, see \cite[Section~1.4]{CvGKTDegenerateContributions}.
\end{remark}

\begin{corollary}\label{lem: boundary indexed by types} The irreducible boundary divisors in $\Mlog$ are indexed by combinatorial types of stable tropical maps to $\RR_{\scriptsize{\geqslant 0}}$ with tangency profile $\upalpha$, whose associated tropical moduli cone is one-dimensional.
\end{corollary}

\begin{proof} Since $\varphi$ is an isomorphism, in particular it identifies the sets of rays.\end{proof}

We let $N(\upalpha)$ denote the number of irreducible boundary divisors in $\Mlog$. By the preceding result, this equals the number of one-dimensional combinatorial types of stable tropical map. This set can be described quite explicitly.

\begin{proposition} \label{prop: comb types of boundary divisors} The combinatorial types of stable tropical maps with one-dimensional tropical moduli fall into the following families
\[
\begin{tikzpicture}
\draw[fill=black] (0,0) circle[radius=2pt];
\draw[blue] (0,0) node[above]{\scriptsize$d_0$};
\draw (0,0) node[below]{\scriptsize$C_0$};

\draw[fill=black] (-2,0.75) circle[radius=2pt];
\draw[blue] (-2,0.75) node[above]{\scriptsize$d_1$};
\draw (-2,0.75) node[left]{\scriptsize$C_1$};

\draw (0,0) -- (-2,0.75);
\draw[blue] (-1,0.375) node[above]{\scriptsize$m_1$};

\draw (-2,0.15) node{$\vdots$};

\draw[fill=black] (-2,-0.75) circle[radius=2pt];
\draw[blue](-2,-0.75) node[below]{\scriptsize$d_k$};
\draw (-2,-0.75) node[left]{\scriptsize$C_k$};

\draw (-2,-0.75) -- (0,0);
\draw[blue] (-1,-0.375) node[below]{\scriptsize$m_k$};

\draw[fill=blue,blue] (-2,-1.75) circle[radius=2pt];
\draw[blue,->] (-2,-1.75) -- (0,-1.75);

\draw[blue,->] (-1,-1) -- (-1,-1.5);
\draw[blue] (-1,-1.25) node[right]{\scriptsize$\mathsf{f}$};

\draw (-1,-2.5) node{\small{Rocket}};
\end{tikzpicture}\qquad\qquad
\begin{tikzpicture}
\draw[fill=black] (-1,0) circle[radius=2pt];
\draw[blue] (-1,0) node[above]{\scriptsize$d$};
\draw (-1,0) node[left]{\scriptsize$C$};


\draw[fill=blue,blue] (-2,-1.75) circle[radius=2pt];
\draw[blue,->] (-2,-1.75) -- (0,-1.75);

\draw[blue,->] (-1,-1) -- (-1,-1.5);
\draw[blue] (-1,-1.25) node[right]{\scriptsize$\mathsf{f}$};

\draw (-1,-2.5) node{\small{Airborne}};
\end{tikzpicture}\qquad\qquad
\begin{tikzpicture}
	
\draw[fill=black] (1,-0.5) circle[radius=2pt];
\draw[blue] (1,-0.5) node[left]{\scriptsize$d_2$};
\draw (1,-0.5) node[right]{\scriptsize$C_2$};

\draw (1,-0.5) -- (1,0.5);
\draw[fill=black] (1,0.5) circle[radius=2pt];
\draw[blue] (1,0.5) node[left]{\scriptsize$d_1$};
\draw (1,0.5) node[right]{\scriptsize$C_1$};


\draw[fill=blue,blue] (1,-1.75) circle[radius=2pt];
\draw[blue,->] (1,-1.75) -- (2,-1.75);

\draw[blue,->] (1.5,-1) -- (1.5,-1.5);
\draw[blue] (1.5,-1.25) node[right]{\scriptsize$\mathsf{f}$};

\draw (1.5,-2.5) node{\small{Binary}};
\end{tikzpicture}
\]
The distribution of the markings $x_i$ is omitted, but can be arbitrary as long as the resulting map is stable and balanced. For rockets we assume $k \geqslant 1$ and $d_j \geqslant m_j > 0$. The balancing conditions at the vertices are
\begin{alignat*}{2}
C_j \colon & \quad \Sigma_{x_i \in C_j} \upalpha_i = d_j - m_j, \\
C_0 \colon & \quad \Sigma_{x_i \in C_0} \upalpha_i = d_0+\Sigma_{j=1}^k m_j.
\end{alignat*}
For the airborne, the source curve is smooth and mapped entirely inside the divisor. For binaries, we assume
\[ \Sigma_{x_i \in C_j} \upalpha_i = d_j \]
for $j \in \{1,2\}$, in order to guarantee that the connecting edge has weight zero.
\end{proposition}

\begin{proof} If there are two or more vertices mapped into $\RR_{>0}$ then the positions $\mathsf{f}(v)$ provide free parameters, so the tropical moduli has dimension $\geqslant 2$. It follows that there is at most one vertex mapped into $\RR_{>0}$. Given this, it is straightforward to deduce that the only possibilities are the types illustrated above.
\end{proof}

The overwhelming majority of combinatorial types are rockets; in other contexts these have been referred to as combs \cite[Lemma~1.12 and Definition~2.2]{GathmannRelative}.

\subsection{Excision} We establish the main generation result.

\begin{proposition}\label{prop: generated by boundary strata} $\ClQ (\Mlog)$ is generated by classes of boundary divisors.
\end{proposition}

\begin{proof} Consider the dense interior
\[ U = M_{0,\upalpha}(\PP^r|H) = \Mlog \setminus \partial \Mlog \]
parametrising maps from a smooth source curve that do not factor through $H$. We will show that $U$ embeds as an open subset of affine space, and hence has trivial class group.

Consider a map $\PP^1 \to \PP^r$ intersecting $H$ at the markings with tangency profile $\upalpha$. This has a unique logarithmic lift, obtained by equipping the source $\PP^1$ with the divisorial logarithmic structure corresponding to the markings, and the base $\Speck$ with the trivial logarithmic structure. Consequently we may identify $U$ with a locus in the space of ordinary stable maps.

Since $n \geqslant 3$ we may fix the first three markings to be $0,1,\infty$. The moduli for the remaining markings is
\[ (\PP^1 \setminus \{0,1,\infty\})^{n-3} \setminus \Delta\]
where $\Delta$ is the large diagonal. Finally the map $\PP^1 \to \PP^r$ is given by specifying sections $f_0,\ldots,f_r$ of $\OO_{\PP^1}(d)$ up to overall scaling. We do not need to quotient by automorphisms of the source curve; these have been rigidified by fixing the first three markings. The section $f_0$ agrees up to scaling with
\[ f_0 = \prod_{i=1}^n s_i^{\upalpha_i} \]
where $s_i \in H^0(\PP^1,\OO_{\PP^1}(1))$ is a section cutting out the marking $x_i$. We fix one such section $f_0$. The remaining moduli is the choice of sections $f_1,\ldots,f_r$. There is no need to quotient by the overall scaling; this has been rigidified by fixing $f_0$. The moduli is the dense open subset of
\[ H^0(\PP^1,\OO_{\PP^1}(d))^{\oplus r} \]
consisting of sections $(f_1,\ldots,f_r)$ such that the linear system spanned by $f_0,f_1,\ldots,f_r$ is basepoint-free. We conclude that $U$ is an open subset of
\[ \left((\PP^1 \setminus \{0,1,\infty\})^{n-3} \setminus \Delta \right) \times H^0(\PP^1,\OO_{\PP^1}(d))^{\oplus r} \]
which is itself open in an affine space. It follows that $\ClQ(U)=0$. By the excision sequence for Chow groups \cite[1.8]{FultonBig} we obtain
\[ \bigoplus_D \Q \cdot D \to \ClQ (\Mlog) \to \ClQ(U) = 0 \]
where the direct sum is over irreducible components of the boundary $\Mlog \setminus U$. This proves that $\ClQ(\Mlog)$ is generated by classes of boundary divisors, as claimed.
\end{proof}

\begin{remark}\label{rmk: absolute stable maps not generated by boundary divisors}
The above result contrasts with \cite[Lemma~1.1.1]{PandPic} where the interior of $\olM_{0,n}(\PP^r,d)$ contributes nontrivially to the class group. Recall from Proposition~\ref{prop: comb types of boundary divisors} that the boundary of $\Mlog$ contains the airborne divisor, over which the source curve is smooth. Conceptually, the logarithmic moduli space therefore has larger boundary, and hence smaller interior, than the ordinary moduli space, explaining the discrepancy.

As we add additional hyperplanes to the logarithmic structure, the interior grows yet smaller; when we reach the full toric boundary, the interior becomes	very affine \cite{RanganathanSkeletons1}.
\end{remark}

\subsection{Enumerating tropical types: maximal contact case} \label{sec: enumerating boundary strata} This section is logically independent of the rest of the paper. It illustrates how to work with combinatorial types in practice.

Set $\upalpha=(d,0,\ldots,0)$ so that all tangency is concentrated at the marking $x_1$. We enumerate the combinatorial types with one-dimensional tropical moduli in this setting. 

\begin{proposition} For $\upalpha=(d,0,\ldots,0)$ of length $n$ we have
\[ N(\upalpha) = 2^{n-1} - n + \sum_{k=1}^d \sum_{k_1=0}^{\min(k,n-1)} \sum_{d_1=0}^{d-k} \sum_{d_2=0}^{d-k-d_1} \left( \sum_{a=0}^{k_1} (-1)^{k_1+a} \dfrac{(a+1)^{n-1}}{a!(k_1-a)!} \right)\! {d_1+k_1-1 \choose k_1-1} p_{k-k_1}(d_2+k-k_1)\]
where $p_k(m)$ is the number of unordered partitions of $m$ into $k$ positive parts, which equals the coefficient of $t^m$ in the power series expansion of
\[ \dfrac{t^k}{(1-t)(1-t^2) \cdots (1-t^k)}.\]
In the above formula the $k_1=0$ instances of the binomial coefficient are defined as follows
\[ {e \choose -1} =
\begin{cases}
1 \quad \text{if } e = -1 \\
0 \quad \text{if } e \geqslant 0
 \end{cases}\]
 and the $k_1=k$ instance of the partition function is
 \[ p_{0}(e) = \begin{cases}
 1  \quad \text{if } e=0 \\
 0 \quad \text{otherwise.}
 \end{cases}\]
\end{proposition}

\begin{proof} We work through the taxonomy of Proposition~\ref{prop: comb types of boundary divisors}. There is always a single airborne type. By the shape of $\upalpha$, any binary type must have degree $d$ on the component containing $x_1$ and degree $0$ on the other component. By stability, there are precisely
\[ 2^{n-1}-n \]
of these. It remains to enumerate the rocket types. We must have $x_1 \in C_0$. The remaining markings can be distributed arbitrarily, except for the special case $k=1,d_0=0$ where by stability there must be at least one other marking on $C_0$. We will allow for arbitrary marking distributions, and then subtract off at the end to account for this special case. 

We enumerate the choices required to determine a rocket type. For integers $p\leqslant q$ we write $[p,q]=\{p,\ldots,q\}$.  We first choose
\begin{equation} \label{eqn: choice k} k \in [1,d] \end{equation}
the number of external components of the source curve. Note that $k \leqslant d$ because every external component has degree $\geqslant 1$. We next choose
\begin{equation} \label{eqn: choice k1} k_1 \in [0,\min(k,n-1)] \end{equation}
the number of external components which carry at least one marking. We must distribute the markings $x_2,\ldots,x_n$ between these external components and the internal component $C_0$.

Letting $T(p,q)$ denote the number of surjective functions $[1,p] \to [1,q]$, the total number of such distributions is
\[ \left( T(n-1,k_1+1) + T(n-1,k_1) \right)/k_1! \]
where the two terms count distributions which do have, respectively do not have, markings assigned to $C_0$. The inclusion-exclusion principle gives
\[ T(p,q) = \sum_{a=1}^q (-1)^{q+a} {q \choose a} a^p \]
from which we compute
\[ \left( T(n-1,k_1+1) + T(n-1,k_1) \right)/k_1! = \sum_{a=0}^{k_1} (-1)^{k_1+a} \dfrac{(a+1)^{n-1}}{a! (k_1-a)!}. \]
We now consider the distribution of the degree. Each external component has degree $\geqslant 1$ and so the $k_1$ external components carrying markings have total degree $\geqslant k_1$. We choose
\begin{equation} \label{eqn: choice d1} d_1 \in [0,d-k] \end{equation}
for the total \emph{additional} degree on these components. This may be distributed arbitrarily; we must count ordered, non-negative partitions of $d_1$ into $k_1$ parts. A stars and bars argument shows that the number of such partitions is
\[ {d_1 + k_1 - 1 \choose k_1 -1}.\]
Finally, consider the $k-k_1$ external components which do not carry any markings. Again the total degree on these must be $\geqslant k-k_1$ and we choose
\begin{equation} \label{eqn: choice d2} d_2 \in [0,d-k-d_1] \end{equation}
for the total additional degree. Again this may be distributed arbitrarily, but since the components carry no markings we must count \emph{unordered}, non-negative partitions of $d_2$ into $k-k_1$ parts. This is equivalent to counting unordered, positive partitions of $d_2+k-k_1$ into $k-k_1$ parts, of which there are precisely
\[ p_{k-k_1}(d_2+k-k_1).\]
Summarising, we have made arbitrary choices at \eqref{eqn: choice k}, \eqref{eqn: choice k1}, \eqref{eqn: choice d1}, \eqref{eqn: choice d2} and explicitly enumerated all other choices. This gives the total number of rocket types as
\[ \sum_{k=1}^d \sum_{k_1=0}^{\min(k,n-1)} \sum_{d_1=0}^{d-k} \sum_{d_2=0}^{d-k-d_1} \left( \sum_{a=0}^{k_1} (-1)^{k_1+a} \dfrac{(a+1)^{n-1}}{a!(k_1-a)!} \right)\! {d_1+k_1-1 \choose k_1-1} p_{k-k_1}(d_2+k-k_1).\]
Combining this with the number of airborne and binary types, and subtracting $1$ for the overcounting of the unstable rocket type, we conclude the result.
\end{proof}

\section{Test curves} \label{sec: test curves}

\noindent We come to the main construction: an assembly line for the manufacture of test curves in the moduli space. The geometry of algebraic surfaces plays a fundamental role.

Fix an irreducible boundary divisor $D$ in the space of stable logarithmic maps, indexed by a combinatorial type of stable tropical map of the form

\[
\begin{tikzpicture}

\draw[fill=black] (0,0) circle[radius=2pt];
\draw[blue] (0,0) node[above]{\scriptsize$d_0$};
\draw (0,0) node[below]{\scriptsize$C_0$};

\draw[fill=black] (-2,0.75) circle[radius=2pt];
\draw[blue] (-2,0.75) node[above]{\scriptsize$d_1$};
\draw (-2,0.75) node[left]{\scriptsize$C_1$};

\draw (0,0) -- (-2,0.75);
\draw[blue] (-1,0.375) node[above]{\scriptsize$m_1$};

\draw (-2,0.15) node{$\vdots$};

\draw[fill=black] (-2,-0.75) circle[radius=2pt];
\draw[blue](-2,-0.75) node[below]{\scriptsize$d_k$};
\draw (-2,-0.75) node[left]{\scriptsize$C_k$};

\draw (-2,-0.75) -- (0,0);
\draw[blue] (-1,-0.375) node[below]{\scriptsize$m_k$};

\draw[fill=blue,blue] (-2,-1.75) circle[radius=2pt];
\draw[blue,->] (-2,-1.75) -- (0,-1.75);

\draw[blue,->] (-1,-1) -- (-1,-1.5);
\draw[blue] (-1,-1.25) node[right]{\scriptsize$\mathsf{f}$};

\end{tikzpicture}
\]
This is a divisor of rocket type (see Propositon~\ref{prop: comb types of boundary divisors}). The distribution of markings is suppressed in the above picture, but understood to be a fixed partition 
\[ A_0 \sqcup A_1 \sqcup \ldots \sqcup A_k = \{ 1, \ldots, n\} \]
such that the resulting tropical map is stable and balanced. Note that $d_j \geqslant m_j > 0$. We give a process for constructing test curves
\[ \PP^1 \to \Mlog \]
which pass through $D$ at $0 \in \PP^1$ and intersect the boundary of the moduli space transversely and away from codimension-$2$ strata. Such a test curve corresponds to a family of maps
\begin{equation} \label{eqn: general family of maps}
\begin{tikzcd}
\Ccal \ar[r,"f"] \ar[d,"\pi"] & \PP^r \\
\PP^1 & 
\end{tikzcd}
\end{equation}
with $\pi^{-1}(0)$ of the shape specified by the above combinatorial type. We have taken $D$ to be a rocket divisor, but the process below can be modified (in fact, simplified) to deal with airborne and binary divisors.

\subsection{Singular total space} 
Before giving the construction, we highlight a new phenomenon which appears in the logarithmic setting: we cannot guarantee that the total space $\Ccal$ is smooth with all marking and special divisors intersecting transversely.

Indeed, suppose we have such a family \eqref{eqn: general family of maps} with $\pi^{-1}(0)$ giving an element of $D$. In a neighbourhood of $\pi^{-1}(0)$ there is an identification of effective Cartier divisors
\[ f^\star H = \Sigma_{i=1}^n \upalpha_i x_i + w C_0 \]
for some $w > 0$. For $j \in \{1,\ldots, k\}$ we then have
\[ m_j = d_j - \Sigma_{x_i \in C_j} \upalpha_i = f^\star H \cdot C_j - \Sigma_{x_i \in C_j} \upalpha_i = \left(\Sigma_{x_i \in C_j} \upalpha_i + w C_0 C_j \right) - \Sigma_{x_i \in C_j} \upalpha_i = wC_0C_j. \]
If $\Ccal$ is smooth with all marking and special divisors intersecting transversely, then $C_0 C_j=1$ for all $j$ and we conclude $m_1=\ldots=m_k=w$. This is certainly not the case for all $D$.

In the following construction, the surface $\Ccal$ will have cyclic quotient singularities arising from non-reduced blowups. This is consistent with the structure of logarithmic curves; the monomial sheaves encode local node smoothings of the form $xy=t^c$ whose total spaces are singular toric surfaces.

\subsection{Arranging the markings}\label{sec: arranging the markings}

We begin the construction. Consider the trivial family
\[ \pi_1 \colon \PP^1 \times \PP^1 \to \PP^1\]
and note that sections of $\pi_1$ correspond to smooth divisors in the linear systems $|\OO_{\PP^1 \times \PP^1}(b,1)|$ with $b \geqslant 0$. Working explicitly with homogeneous coordinates, we can choose smooth divisors
\begin{equation} \label{eqn: defn of xi} x_i \in |\OO_{\PP^1 \times \PP^1}(b_i,1)| \end{equation}
for $i \in \{1,\ldots,n\}$, satisfying the following specifications.

\begin{specification}[Central fibre] \label{conditions central fibre} For those $j \in \{ 1,\ldots,k\}$ with $A_j \neq \emptyset$, the intersection
\begin{equation} \label{eqn: intersection of A_i divisors in central fibre} \pi^{-1}(0) \cap \bigcap_{i \in A_j} x_i \end{equation}
is nonempty, and hence consists of a single point. In contrast, for all $i \in A_0$ the point $\pi^{-1}(0) \cap x_i$ is disjoint from the other $x_i$.
\end{specification}

\begin{specification}[Genericity] \label{conditions genericity} Away from $\pi^{-1}(0)$, there is no point where $3$ or more of the $x_i$ intersect, and in each fibre of $\pi$ there is at most one pairwise intersection point. At every intersection point, including those on the central fibre, the intersections are pairwise transverse and no $x_i$ is tangent to the fibre of $\pi$.
\end{specification}

For those $j \in \{1,\ldots,k\}$ with $A_j \neq \emptyset$, let $p_j$ denote the intersection point \eqref{eqn: intersection of A_i divisors in central fibre}. If $A_j=\emptyset$, let $p_j$ be an arbitrary point of $\pi^{-1}(0)$ disjoint from the $x_i$. Denote by $q_1,\ldots,q_l$ the intersection points of the $x_i$ away from $\pi^{-1}(0)$. The surface $\PP^1 \times \PP^1$ will be blown up in the points $p_1,\ldots,p_k,q_1,\ldots,q_l$ to produce the desired curve family.

\subsection{Blowing up} \label{sec: blowing up}
Define multiplicities
\begin{align*}
w & = \operatorname{lcm}(m_1,\ldots,m_k), \\
w_j & = w/m_j \quad \text{for } j \in \{1,\ldots,k\}.	
\end{align*}
Perform the non-reduced blowup of $\PP^1 \times \PP^1$ at the points $p_1,\ldots,p_k$ with multiplicities $w_1,\ldots,w_k$ in the first factor and $1,\ldots,1$ in the second factor. Then perform the ordinary blowup of the resulting surface at the points $q_1,\ldots, q_l$. The output is a singular surface $\Ccal$ together with a projection
\[ \pi \colon \Ccal \to \PP^1 \]
which is a flat family of nodal curves. 

\subsection{Special fibres} The family $\pi$ has singular fibres precisely over the images of the special points $p_1,\ldots,p_k,q_1,\ldots,q_l$. The fibre over $0 \in \PP^1$ is the reduced union of the following $\Q$-Cartier divisors:
\begin{itemize}
\item $C_0$, the strict transform of $\pi^{-1}(0) \subseteq \PP^1 \times \PP^1$.
\item $C_1,\ldots,C_k$, the reduced exceptional divisors over $p_1,\ldots,p_k$.
\end{itemize}
These divisors satisfy $C_0 C_j = 1/w_j$ and $C_{j_1} C_{j_2}=0$. Note that both $wC_0$ and $w_j C_j$ are Cartier. Moreover we have
\begin{align*}
w C_0 C_j & = w/w_j = m_j, \\
w C_0^2 & = wC_0 (-C_1-\ldots-C_k) = -(m_1+\ldots+m_k).
\end{align*}
The fibre over the image of $q_j$ is a nodal curve with two components: $D_j$ the strict transform of the original fibre, and $E_j$ the exceptional divisor of the blowup. The latter intersects the strict transforms of the two marking sections $x_{j_1},x_{j_2}$ which intersected in $\PP^1 \times \PP^1$ at $q_j$. We will construct a map with degree $d$ on $D_j$ and degree $0$ on $E_j$. Orienting the node from $D_j$ to $E_j$, the desired tangency order at the node is given by the balancing condition as
\[ u_j \colonequals \upalpha_{j_1} + \upalpha_{j_2}. \]
Finally, let $x_1,\ldots,x_n \subseteq \Ccal$ denote the strict transforms of the marking sections $x_1,\ldots,x_n \subseteq \PP^1 \times \PP^1$. By Specification~\ref{conditions genericity}, these are pairwise disjoint and contained in the smooth locus of $\pi$.

The resulting curve family is illustrated below. Divisors coloured red are those which will be mapped into $H$. The role of the smooth fibres $B$ will be explained shortly.

\[\includegraphics{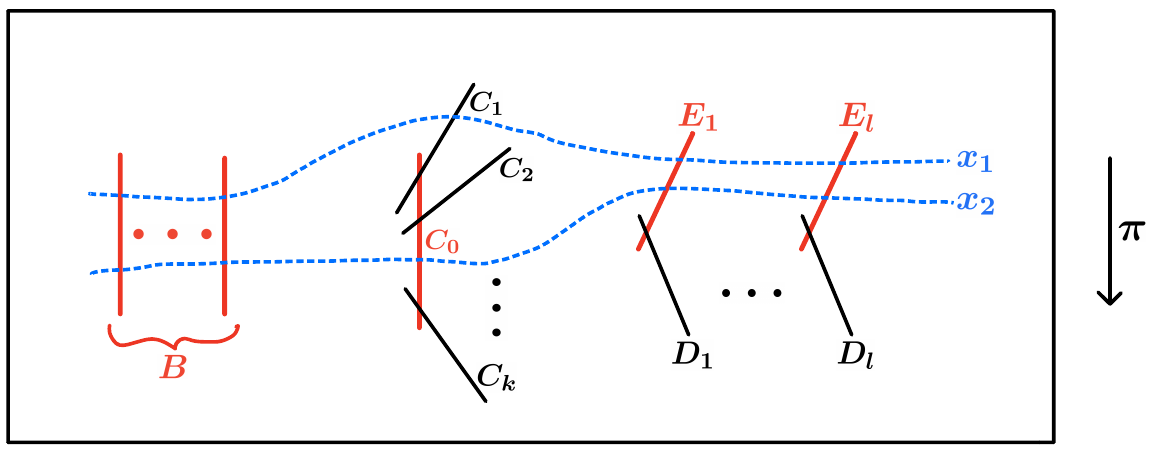}\]

\subsection{Constructing $f_0$} \label{sec: constructing f0} It remains to construct the map $\Ccal \to \PP^r$. This is achieved by constructing appropriate sections $f_0,f_1,\ldots,f_r$ of the following line bundle
\begin{equation}  \label{eqn: line bundle L} L \colonequals \OO_\Ccal(\Sigma_{i=1}^n \upalpha_i x_i+ wC_0 +\Sigma_{j=1}^l u_j E_j) \otimes \pi^\star \OO_{\PP^1}(N) \end{equation}
for $N \geqslant 0$ to be determined. We verify that $L$ has the desired multi-degree on every special fibre:
\begin{align*} \deg L|_{C_0} & = \Sigma_{x_i \in C_0} \upalpha_i + w C_0^2 = \Sigma_{x_i \in C_0}\upalpha_i - (m_1+\ldots+m_k) = d_0,\\
\deg L|_{C_j} & = \Sigma_{x_i \in C_j} \upalpha_i + w C_0 C_j = \Sigma_{x_i \in C_j}\upalpha_i + m_j = d_j,	\\
\deg L|_{D_j} & = \Sigma_{x_i \in D_j} \upalpha_i + u_j = d, \\
\deg L|_{E_j} & = \upalpha_{j_1} + \upalpha_{j_2} - u_j = 0.
\end{align*}
Assume without loss of generality that $H=H_0 \subseteq \PP^r$ so that the section $f_0$ governs the tangency with respect to $H$. Fix a generic section of $\OO_{\PP^1}(N)$ and let $B \subseteq \Ccal$ denote the vanishing locus of the pullback; this is a union of smooth fibres of $\pi$. Define
\[ f_0 \in H^0(\Ccal,L) \]
as the product of the standard section cutting out the effective Cartier divisor on the first factor of \eqref{eqn: line bundle L}, together with the section cutting out $B$ on the second factor of \eqref{eqn: line bundle L}.

We verify that $f_0$ has the desired behaviour. On a general fibre of $\pi$ the vanishing locus of $f_0$ is
\[ \Sigma_{i=1}^n \upalpha_i x_i \]
which gives the desired tangency profile. Focusing on $\pi^{-1}(0)$ we see that $f_0$ vanishes identically on $C_0$ and has the desired vanishing order at every marking in $A_1 \sqcup \ldots \sqcup A_k$. The vanishing order at the node $C_0 \cap C_j$ is given by
\[ w C_0 C_j = w/w_j = m_j \]
as required. Hence $f_0$ has the desired behaviour on the central fibre. On the other singular fibres, $f_0$ vanishes identically on $E_j$ (assuming $u_j \neq 0$) and its vanishing order at the node $D_j \cap E_j$ is
\[ u_j D_j E_j = u_j \]
as required. Finally $f_0$ vanishes identically on the smooth special fibres which constitute $B$. These fibres give points in the airborne boundary divisor (see Proposition~\ref{prop: comb types of boundary divisors} and Section~\ref{sec: defn alien divisor}). This verifies the behaviour of $f_0$.

\subsection{Constructing $f_1,\ldots,f_r$} \label{sec: constructing f1 to fm} The remaining sections of $L$ have no constraints, and may be chosen arbitrarily as long as the resulting tuple $[f_0,f_1,\ldots,f_r]$ has no basepoints. Recall that
\[ L = \OO_\Ccal(\Sigma_{i=1}^n \upalpha_i x_i+ wC_0 +\Sigma_{j=1}^l u_j E_j) \otimes \pi^\star \OO_{\PP^1}(N) \]
where $N \geqslant 0$ is fixed but arbitrary. The following lemma constructs the desired sections of $L$, and in so doing determines a lower bound for $N$.
\begin{lemma} \label{lem: BP free curves} Define the integer
\[ M \colonequals d_0 + \Sigma_{j=1}^k d_j w_j - \Sigma_{i=1}^n \upalpha_i b_i - w.\]
If $N \geqslant \max(M,0)$ then there exist curves
\[ F_1,\ldots, F_r \in |L| \]
such that the intersection $V(f_0) \cap F_1 \cap \ldots \cap F_r \subseteq \Ccal$ is empty.
\end{lemma}

\begin{proof}
For $i \in \{1,\ldots,r\}$ and $j \in \{1,\ldots,k\}$ choose a curve
\[ G_{ij} \in |\OO_{\PP^1 \times \PP^1}(w_j,1)| \]
that passes through the point $p_j$ with tangency order $w_j$ with respect to the horizontal line $\PP^1 \times \{ \pi_2(p_j) \}$. For $i \in \{1,\ldots,r\}$ and $j=0$ choose an arbitrary curve
\[ G_{i0} \in |\OO_{\PP^1 \times \PP^1}(1,1)|. \]
Generic choices of such curves satisfy the following.

\begin{specification}\label{assumption on Gij} For $i \in \{1,\ldots,r\}$ and $j\in \{0,1,\ldots,k\}$ the curve $G_{ij}$ does not pass through the points $p_1,\ldots,\widehat{p_j},\ldots,p_k,q_1,\ldots,q_l$. At $p_j$, the equations defining the curves
\[ G_{1j},\ldots,G_{rj} \]
have distinct partial derivatives of order $w_j$ with respect to the horizontal coordinate. The points
\[ G_{10} \cap \pi^{-1}(0),\ldots,G_{r0} \cap \pi^{-1}(0)\]
are all distinct. Away from $\pi^{-1}(0)$ the $G_{ij}$ intersect only pairwise and away from $\varphi(V(f_0))$, where $\varphi \colon \Ccal \to \PP^1 \times \PP^1$ is the blowdown.
\end{specification}

For $i \in \{1,\ldots,r\}$ define the (non-reduced) Cartier divisor
\begin{equation} \label{eqn: defn of Gi} G_i \colonequals \sum_{j=0}^k d_j G_{ij} \in |\OO_{\PP^1 \times \PP^1}(d_0 + \Sigma_{j=1}^k d_j w_j,d)| \end{equation}
and consider the strict transform
\[ \widetilde{G}_i \subseteq \Ccal.\]
Specification~\ref{assumption on Gij} ensures that the intersection $\varphi (V(f_0)) \cap G_1 \cap \ldots \cap G_r$ consists precisely of the points $p_1,\ldots,p_k$. Distinctness of the higher partial derivatives ensures that $G_1,\ldots,G_r$ are separated by the blowup. Since $r \geqslant 2$ we conclude that
\[ V(f_0) \cap \widetilde{G}_1 \ldots \cap \widetilde{G}_r = \emptyset. \]
It remains to compare $\OO_\Ccal(\widetilde{G}_i)$ and $L$. The tangency specification on $G_{ij}$ gives
\begin{align*} \widetilde{G}_{ij} & = \varphi^\star G_{ij} - w_j C_j \\
\widetilde{G}_{i0} & = \varphi^\star G_{i0}
\end{align*}
and so we have
\begin{equation} \label{eqn: Gi vs Gitilde} \widetilde{G}_i = \varphi^\star G_i - \Sigma_{j=1}^k d_j w_j C_j.\end{equation}
Now recall the arrangement of markings from Section~\ref{sec: arranging the markings}. Combining \eqref{eqn: defn of xi} and \eqref{eqn: defn of Gi} together with $\Sigma_{i=1}^n \upalpha_i = d$, we have
\begin{equation} \label{eqn: Ftilde in terms of markings} \OO_{\PP^1 \times \PP^1}(G_i) = \OO_{\PP^1 \times \PP^1}(d_0+\Sigma_{j=1}^k d_j w_j,d) = \OO_{\PP^1 \times \PP^1}(\Sigma_{i=1}^n \upalpha_i x_i) \otimes \pi_1^\star \OO_{\PP^1}(M^\prime) \end{equation}
where $M^\prime=d_0+\Sigma_{j=1}^k d_j w_j - \Sigma_{i=1}^n \upalpha_i b_i$. Combining \eqref{eqn: Gi vs Gitilde} and \eqref{eqn: Ftilde in terms of markings} we obtain
\begin{equation} \label{eqn: second expression for Fitilde} \OO_{\Ccal}(\widetilde{G}_i) = \OO_{\Ccal}(\varphi^\star \Sigma_{i=1}^n \upalpha_i x_i - \Sigma_{j=1}^k d_j w_j C_j) \otimes \pi^\star \OO_{\PP^1}(M^\prime).\end{equation}
To control the pullback of the marking divisors, we impose the following specification. This can always be achieved by increasing the vertical degrees $b_i$.

\begin{specification} \label{assumption marking tangent} Consider the marking divisor $x_i \subseteq \PP^1 \times \PP^1$ given in \eqref{eqn: defn of xi}. If $x_i$ passes through the point $p_j \in \pi^{-1}(0)$ (i.e. if $i \in A_j$), then at this point the curve $x_i$ has tangency order $w_j$ with respect to the horizontal line.
\end{specification}

This gives the following comparison between $x_i \subseteq \PP^1 \times \PP^1$ and its strict transform $x_i \subseteq \Ccal$
\begin{equation*} \label{eqn: strict transform of xi}
\varphi^\star x_i =
\begin{cases}
x_i + w_j C_j + \sum_{q_j \in x_i} E_j \qquad \text{if $i \in A_{j\neq 0}$} \\\medskip
x_i + \sum_{q_j \in x_i} E_j \qquad \qquad \quad \, \, \, \text{if $i \in A_0$}
\end{cases}
\end{equation*}
from which we obtain
\[ \varphi^\star \Sigma_{i=1}^n \upalpha_i x_i = \Sigma_{i=1}^n \upalpha_i x_i + \Sigma_{j=1}^k \left( \Sigma_{i \in A_j} \upalpha_i \right) w_j C_j + \Sigma_{j=1}^l u_j E_j.\]
Substituting this into \eqref{eqn: second expression for Fitilde} we obtain
\begin{align*} 
\OO_{\Ccal}(\widetilde{G}_i) & = \OO_{\Ccal}(\Sigma_{i=1}^n \upalpha_i x_i + \Sigma_{j=1}^k \left( -d_j + \Sigma_{i \in A_j} \upalpha_i \right) w_j C_j + \Sigma_{j=1}^l u_j E_j) \otimes \pi^\star \OO_{\PP^1}(M^\prime)\\
& = \OO_{\Ccal}(\Sigma_{i=1}^n \upalpha_i x_i + \Sigma_{j=1}^k (-m_j) w_j C_j + \Sigma_{j=1}^l u_j E_j) \otimes \pi^\star \OO_{\PP^1}(M^\prime).
\end{align*}
For all $j \in \{1,\ldots,k\}$ we have $m_j w_j=w$. Moreover since the fibre $\pi^{-1}(0)$ is the reduced union $C_0+C_1+\ldots+C_k$ we have
\[ \OO_{\Ccal}(\Sigma_{j=1}^k (-m_j w_j) C_j) = \OO_{\Ccal}(-w \Sigma_{j=1}^k C_j) = \OO_{\Ccal}(wC_0) \otimes \pi^\star \OO_{\PP^1}(-w).\]
Setting $M=M^\prime-w$ we obtain
\begin{align*}
\OO_{\Ccal}(\widetilde{G}_i) & = \OO_{\Ccal}(\Sigma_{i=1}^n \upalpha_i x_i + wC_0 + \Sigma_{j=1}^l u_j E_j) \otimes \pi^\star \OO_{\PP^1}(M)\\
& = L \otimes \pi^\star \OO_{\PP^1}(M-N).
\end{align*}
Since $N \geqslant M$ we may choose a general fibre $T_i \in |\pi^\star \OO_{\PP^1}(1)|$ and consider the effective divisor
\[ F_i \colonequals \widetilde{G}_i + (N-M) T_i.\]
If the $T_i$ are chosen generically then since $r \geqslant 2$ the $F_i$ still satisfy $V(f_0) \cap F_1 \cap \ldots \cap F_r = \emptyset$. Moreover we now have
\[ \OO_{\Ccal}(F_i) = L\]
as required.
\end{proof}

\subsection{Logarithmic structures} Finally, equip the surface $\Ccal$ with the divisorial logarithmic structure corresponding to the reduced union of: the marking divisors $x_i$, the singular fibres of $\pi$, and the smooth fibres $B$. Equip the base $\PP^1$ with the divisorial logarithmic structure corresponding to the reduced image of the singular fibres of $\pi$ and the smooth fibres $B$. The logarithmic enhancements of the morphisms $f$ and $\pi$ are then unique, producing a diagram of logarithmic schemes
\[
\begin{tikzcd}
(\Ccal,M_{\Ccal}) \ar[r,"f"] \ar[d,"\pi"]	 & (\PP^r|H) \\
(\PP^1,M_{\PP^1}). &
\end{tikzcd}
\]
This family of stable logarithmic maps is automatically basic \cite[Section~1.5]{GrossSiebertLog}, as it only intersects codimension~$1$ strata which all have basic monoid $\N$. We thus obtain the desired test curve
\[ \PP^1 \to \Mlog \]
and this completes the construction.

\section{Relations} \label{sec: relations}

\subsection{Pulling back divisors} \label{sec: pulling back divisors} Given a test curve
\[ \PP^1 \to \Mlog \]
we wish to pull back elements of the class group of $\Mlog$. This requires some care, since the divisors may not be Cartier and the test curve may not be a regular embedding (or even a closed embedding). Let
\[ W \subseteq \Mlog \]
denote the complement of the boundary strata of codimension $\geqslant 2$. Since $\Mlog$ is logarithmically smooth, it follows that $W$ is smooth (because all normal toric varieties of dimension $\leqslant 1$ are smooth). By the construction in Section~\ref{sec: test curves}, the test curve factors through this open subset
\[ \PP^1 \to W \hookrightarrow \Mlog.\]
Given an element in the class group of $\Mlog$, we first pull it back along the flat morphism $W \hookrightarrow \Mlog$. Since $W$ is smooth the resulting divisor is Cartier, so we can pull it back along the morphism $\PP^1 \to W$.

Equivalently, $\PP^1 \to W$ is l.c.i. (since both source and target are smooth) and $W \hookrightarrow \Mlog$ is l.c.i. (since it is in fact smooth). The composite $\PP^1 \to \Mlog$ is therefore l.c.i. and induces a pullback on Chow groups \cite[Section~6.6]{FultonBig}.

We will refer to the above process as "intersecting a divisor on $\Mlog$ with a test curve". The result agrees with the intuitive expectation.

\subsection{Terrestrial, airborne and alien divisors} \label{sec: defn alien divisor} The test curves constructed above intersect the boundary of $\Mlog$ at finitely many points away from the central point $0$. These additional intersections occur only at a special class of boundary divisors, which play a key role in the forthcoming analysis.

\begin{definition} \label{def: alien and terrestrial} The \textbf{airborne} divisor is the boundary divisor in $\Mlog$ whose generic point parametrises a smooth curve mapped inside $H$. The combinatorial type is
\[
\begin{tikzpicture}
	
\draw[fill=black] (-1,-0.5) circle[radius=2pt];
\draw[blue] (-1,-0.5) node[above]{\scriptsize$d$};

\draw[->] (-1,-0.5) -- (-0.5,-0.25);
\draw[->] (-1,-0.5) -- (-0.5,-0.75);

\draw[fill=blue,blue] (-2,-1.75) circle[radius=2pt];
\draw[blue,->] (-2,-1.75) -- (0,-1.75);

\draw[blue,->] (-1,-1) -- (-1,-1.5);
\draw[blue] (-1,-1.25) node[right]{\scriptsize$\mathsf{f}$};

\end{tikzpicture}
\]
The \textbf{alien} divisors are the boundary divisors in $\Mlog$ whose generic point parametrises a curve with two irreducible components, one of which contains two markings and has degree $0$, the other of which contains all the other markings and has degree $d$. Since $d \geqslant 1$ there are precisely ${n \choose 2}$ of these, and they are denoted $D_{ij}$ where $1 \leqslant i < j \leqslant n$ index the two markings. Their combinatorial types are:
\begin{center}
\begin{tikzpicture}
	
\draw[fill=black] (-2,-0.5) circle[radius=2pt];
\draw[blue] (-2,-0.5) node[above]{\scriptsize$d$};

\draw (-2,-0.5) -- (-1.5,-0.5);

\draw[fill=black] (-1.5,-0.5) circle[radius=2pt];
\draw[blue] (-1.5,-0.5) node[above]{\scriptsize$0$};

\draw[->] (-1.5,-0.5) -- (-1,-0.25);
\draw (-1.1,-0.25) node[right]{\scriptsize$x_i$};

\draw[->] (-1.5,-0.5) -- (-1,-0.75);
\draw (-1.1,-0.75) node[right]{\scriptsize$x_j$};

\draw[fill=blue,blue] (-2,-1.75) circle[radius=2pt];
\draw[blue,->] (-2,-1.75) -- (-1,-1.75);

\draw[blue,->] (-1.5,-1) -- (-1.5,-1.5);
\draw[blue] (-1.5,-1.25) node[right]{\scriptsize$\mathsf{f}$};

\draw (-1.5,-2.5) node{$\upalpha_i+\upalpha_j > 0$};
	
\draw[fill=black] (1,-1) circle[radius=2pt];
\draw[blue] (1,-1) node[left]{\scriptsize$d$};

\draw (1,-1) -- (1,-0.5);
\draw[fill=black] (1,-0.5) circle[radius=2pt];
\draw[blue] (1,-0.5) node[left]{\scriptsize$0$};

\draw[->] (1,-0.5) -- (1.25,0);
\draw (1.25,0) node[above]{\scriptsize$x_i$};

\draw[->] (1,-0.5) -- (0.75,0);
\draw (0.75,0) node[above]{\scriptsize$x_j$};

\draw[fill=blue,blue] (1,-1.75) circle[radius=2pt];
\draw[blue,->] (1,-1.75) -- (2,-1.75);

\draw[blue,->] (1.5,-1) -- (1.5,-1.5);
\draw[blue] (1.5,-1.25) node[right]{\scriptsize$\mathsf{f}$};

\draw (1.5,-2.5) node{$\upalpha_i+\upalpha_j = 0$};

\end{tikzpicture}
\end{center}

\noindent Boundary divisors which are neither airborne nor alien are referred to as \textbf{terrestrial}. Note that this taxonomy of boundary divisors differs from the one given in Proposition~\ref{prop: comb types of boundary divisors}; aliens and terrestrials may be either rockets or binaries.
\end{definition}

The test curves of Section~\ref{sec: test curves} are constructed in such a way that they do not intersect any terrestrial divisor away from the central point $0$.

\subsection{Linear independencies} We now leverage our control over test curves to deduce the main theorem.

\begin{proposition} \label{prop: airborne divisor zero coefficient} The airborne divisor has coefficient zero in any linear relation in $\ClQ (\Mlog)$ between boundary divisors.
\end{proposition}

\begin{proof} Take such a linear relation. Construct any test curve in $\Mlog$ using the techniques of Section~\ref{sec: test curves}. We then modify this test curve as follows. Replace $L$ by $L \otimes \pi^\star \OO_{\PP^1}(1)$ and replace every $f_i$ by $f_i t_i$ for $t_i \in H^0(\Ccal,\pi^\star \OO_{\PP^1}(1))$ a section cutting out a fibre of $\pi$. If the $t_i$ are chosen generically then the map $[f_0t_0,\ldots,f_rt_r]$ will be basepoint-free. We obtain a new test curve whose intersection with the boundary is the same as the original test curve, except that it intersects the airborne divisor in one more point than before. Intersecting both test curves with the given linear relation and taking the difference, we conclude that the coefficient of the airborne divisor is zero.
\end{proof}

\begin{proposition} \label{prop: airborne and terrestrial linearly independent set} The union of the terrestrial divisors and the airborne divisor forms a linearly independent subset of $\ClQ(\Mlog)$.
\end{proposition}

\begin{proof} By the previous proposition, it suffices to show that the set of terrestrial divisors forms a linearly independent subset. Suppose we are given a linear relation amongst the terrestrial divisors, and fix an arbitrary terrestrial divisor $D$. The construction in Section~\ref{sec: test curves} produces a test curve which intersects $D$ in a single point and does not intersect any other terrestrial divisors. Intersecting this test curve with the given linear relation, we conclude that the coefficient of $D$ is zero. Since $D$ was arbitrary, this completes the proof.	
\end{proof}

\begin{definition} Let $S \subseteq \ClQ(\Mlog)$ denote the subspace with basis the set of terrestrial and airborne divisors, and denote the quotient
\[  Q \colonequals \ClQ (\Mlog)/S.\]	
\end{definition}

\begin{proposition} \label{prop: dim of Q} The quotient $Q$ is $n$-dimensional. All boundary relations in $Q$ arise by pullback from $\Mbarn$.
\end{proposition}

\begin{proof} By Proposition~\ref{prop: generated by boundary strata}, $Q$ is spanned by the alien divisors. There are precisely ${n \choose 2}$ of these. On the other hand, Proposition~\ref{prop: basis for WDVV} below exhibits a set of
$${n-1 \choose 2} - 1$$ independent relations between the boundary divisors in $\Mbar_{0,n}$. These pull back to relations in $\ClQ(\Mlog)$ which induce relations in the quotient $Q$. We claim that these relations are independent.

Indeed, the divisor $E_{ij} \subseteq \Mbarn$ defined in Appendix~\ref{sec: divisors on Mbarn} pulls back to the class of the alien divisor $D_{ij}$ in $Q$. The same argument as in the proof of Proposition~\ref{prop: basis for WDVV} then shows that the relations are independent. It follows that the dimension of $Q$ is bounded from above:
\[ \dim Q \leqslant {n \choose 2} - \left( {n-1 \choose 2}-1 \right) = n.\]
We claim that these are all the relations in $Q$. For this, we will show that $\dim Q \geqslant n$, by exhibiting a linearly independent subset of size $n$. Consider the following alien divisors
\begin{equation} \label{eqn: linearly independent set in Q} D_{12},D_{13},\ldots,D_{1n},D_{23}.\end{equation}
Suppose that there is a linear relation between these in $Q$, with $D_{ij}$ having coefficient $a_{ij}$. This implies that the corresponding linear combination in $\Mlog$ belongs to the span of the airborne and terrestrial divisors. However, by Proposition~\ref{prop: airborne divisor zero coefficient}, it in fact belongs to the span of the terrestrial divisors, giving the following relation in $\ClQ(\Mlog)$
\begin{equation} \label{eqn: relation divisors in proof Q n-dimensional} a_{12} D_{12} + a_{13} D_{13} + \ldots + a_{1n} D_{1n} + a_{23} D_{23} = \sum_{D} a_D D \end{equation}
where the sum on the right-hand side is over the set of terrestrial divisors.

We now construct test curves. This requires a minor retooling of the assembly line of Section~\ref{sec: test curves}. Namely, we drop Specification~\ref{conditions central fibre} governing the behaviour of the marking divisors on the central fibre. In this way we produce a test curve which does not intersect any terrestrial divisor. This also enables us to drop Specification~\ref{assumption marking tangent}, and as such we may choose the $b_i$ in \eqref{eqn: defn of xi} arbitrarily (previously each $b_i$ was bounded from below by the corresponding $w_j$).

Fix $i \in \{1,\ldots,n\}$ and arrange markings as in Section~\ref{sec: arranging the markings} with
\[ b_i = 1, \quad b_{\neq i} =0.\]
This ensures that $x_i \cdot x_{\neq i} = 1$ whilst $x_j \cdot x_k = 0$ otherwise. The resulting test curve does not intersect any terrestrial divisors, and intersects $D_{jk}$ if and only if $i \in \{j,k\}$. Intersecting against \eqref{eqn: relation divisors in proof Q n-dimensional} we obtain the following equations
\begin{alignat*}{2}
i=1: \qquad & a_{12} + a_{13} + \ldots + a_{1n}\ & = 0 \\
i=2: \qquad & a_{12} + a_{23} & = 0 \\
i=3: \qquad & a_{13} + a_{23} & = 0 \\
i=4: \qquad & a_{14} & = 0 \\
& \qquad \qquad \vdots & \\
i=n: \qquad & a_{1n} & = 0
\end{alignat*}
which are easily solved to give all $a_{ij}=0$. We conclude that the set \eqref{eqn: linearly independent set in Q} is linearly independent, as required. \end{proof}

Combining Propositions~\ref{prop: airborne and terrestrial linearly independent set} and \ref{prop: dim of Q} we conclude the main result.
\begin{theorem}[Theorem~\ref{thm: main introduction}] \label{thm: main} For $r \geqslant 2, n \geqslant 3, d \geqslant 1$, and $\upalpha \vdash d$, the dimension of $\ClQ(\Mlog)$ is
\[ N(\upalpha) - {n-1 \choose 2} + 1.\]
In particular, it does not depend on $r$. A basis is given by the union of the airborne and terrestrial divisors and the subset \eqref{eqn: linearly independent set in Q} of the alien divisors. Moreover, all relations amongst boundary divisors are pulled back from $\Mbarn$.
\end{theorem}

\subsection{Picard group} \label{sec: Picard} The space $\Mlog$ is logarithmically smooth, hence normal. It follows that the Picard group embeds in the class group
\begin{equation} \label{eqn: Picard injects into class} \PicQ \Mlog \hookrightarrow \ClQ \Mlog \end{equation}
as the locally principal divisor classes, see \cite[Theorem~11.38(1)]{GortzWedhorn} and \cite[Remark~II.6.1.12]{HartshorneAG}. Applying Proposition~\ref{prop: generated by boundary strata} we see that the Picard group is generated by Cartier divisors supported on the boundary. These admit the following description.

\begin{lemma} \label{lem: Cartier divisors are PL functions} Boundary Cartier divisors on $\Mlog$ are indexed by integral piecewise-linear functions on the moduli space $\Sigma_{0,\upalpha}$ of stable tropical maps to $\R_{\geqslant 0}$.
\end{lemma}

\begin{proof} By Theorem~\ref{thm: faithful tropicalisation} we have a natural isomorphism $\Sigma \Mlog = \Sigma_{0,\upalpha}$. It is then a general fact about toroidal embeddings that boundary Cartier divisors correspond to integral piecewise-linear functions on the tropicalisation.  
\end{proof}

We arrive at a combinatorial description of the Picard group.

\begin{theorem}[Theorem~\ref{thm: Picard introduction}] \label{thm: Picard} The rational Picard group of $\Mlog$ is generated by divisors corresponding to integral piecewise-linear functions on $\Sigma_{0,\upalpha}$. A complete set of relations between these divisors is obtained by pulling back the WDVV relations from $\Mbarn$.
\end{theorem}

\begin{proof}
The description of the generators follows from Lemma~\ref{lem: Cartier divisors are PL functions}. The description of the relations follows from the injectivity of \eqref{eqn: Picard injects into class} and the corresponding result for the class group (Theorem~\ref{thm: main}). Note that since $\Mbarn$ is smooth, every boundary relation pulls back to a relation between boundary \emph{Cartier} divisors on $\Mlog$.
\end{proof}

\appendix
\section{Divisors on $\Mbarn$} \label{sec: divisors on Mbarn}
\noindent The following facts can all be found in \cite{KeelCurves}. Fix $n \geqslant 3$. The Grothendieck--Knudsen space $\Mbarn$ is a smooth projective variety of dimension $n-3$. Its Picard group has dimension $2^{n-1} - {n-1 \choose 2} - n$ and is generated by the set of boundary divisors, of which there are $2^{n-1} - 1 - n$. It follows that there are
\[ {n-1 \choose 2} - 1 \]
independent relations amongst the boundary divisors. All such relations are obtained by pulling back the WDVV relations from $\ol{M}_{0,4}$. However, there are $2 {n \choose 4}$ of these, so while they generate all relations, they are not independent.

We now exhibit a complete and independent set of boundary relations. This is used in the proof of Proposition~\ref{prop: dim of Q} above. By the preceding discussion, it suffices to exhibit a linearly independent set of relations of size ${n-1 \choose 2} - 1$. We let
\[ \mathcal{R} \begin{pmatrix} a & b \\ c & d \end{pmatrix} \]
denote the WDVV relation
\[ D(ab | cd) = D(ac | bd ).\]
Note that $\mathcal{R}(M) = \mathcal{R}(M^T)$. For $1 \leqslant i < j \leqslant n$ let $E_{ij} \subseteq \Mbarn$ denote the boundary divisor whose generic point parametrises a curve with two components, one of which contains precisely the markings $x_i,x_j$.

\begin{lemma} \label{lem: boundary appears in RM iff row or column} The boundary divisor $E_{ij}$ appears in the relation $\mathcal{R}(M)$ if and only if $(i j)$ appears as a row or a column of $M$.\end{lemma}

\begin{proof} The irreducible components of $D(ab|cd)$ are indexed by splittings of the markings into two parts, with $x_a,x_b$ in the first part and $x_c,x_d$ in the second part. We conclude that $E_{ij}$ is an irreducible component of $D(ab|cd)$ if and only if $\{i,j\}=\{a,b\}$ or $\{i,j\} =\{c,d\}$. The same applies to $D(ac|bd)$, and the result follows.
\end{proof}

With this, the search for a set of linearly independent relations reduces to a game of sudoku.

\begin{proposition} \label{prop: basis for WDVV} The $\mathcal{R}(M)$ in the following table form a basis for the boundary relations in $\Mbarn$.\medskip
\renewcommand{\arraystretch}{0.8}
\[
\begin{tabular}{ c | ccccccccc}
$\mathcal{R}(M)$ & $\begin{pmatrix} 1 & 2 \\ 3 & 4 \end{pmatrix}$ & $\begin{pmatrix} 1 & 3 \\ 4 & 5 \end{pmatrix}$ & $\cdots$ & $\begin{pmatrix} 1 & n\!-\!2 \\ n\!-\!1 & n \end{pmatrix}$ & $\begin{pmatrix} 2 & 3 \\ 4 & 5 \end{pmatrix}$ & $\cdots$ & $\begin{pmatrix} 2 & n\!-\!2 \\ n\!-\!1 & n \end{pmatrix}$ & $\begin{pmatrix} 3 & 4 \\ 5 & 6 \end{pmatrix}$ & $\cdots$
\\
&&&&&&&&& \\ \hline
&&&&&&&&& \\
$E(M)$ & $E_{12}$ & $E_{13}$ & $\cdots$ & $E_{1,n\!-\!2}$ & $E_{23}$ & $\cdots$ & $E_{2,n\!-\!2}$ & $E_{34}$ & $\cdots$
\end{tabular}
\]\medskip
\[
\begin{tabular}{ c | ccccccc}
$\mathcal{R}(M)$ & $\cdots$ & $\begin{pmatrix} 3 & n\!-\! 2 \\ n\!-\! 1 & n \end{pmatrix}$ & $\cdots$ & $\begin{pmatrix} n\!-\!3 & n\!-\!2 \\ n\!-\!1 & n \end{pmatrix}$ & $\begin{pmatrix} n & 1 \\ n\!-\!2 & n\!-\! 1 \end{pmatrix}$ & $\cdots$ & $\begin{pmatrix} n & n\!-\!3 \\ n\!-\!2 & n\!-\! 1 \end{pmatrix}$\\
&&&&&&& \\ \hline
&&&&&&& \\
$E(M)$ & $\cdots$ & $E_{3,n-2}$ & $\cdots$ & $E_{n\!-\!3,n\!-\!2}$ & $E_{1n}$ & $\cdots$ & $E_{n\!-\!3,n}$
\end{tabular}
\]
\end{proposition}

\begin{proof} It is straightforward to check that there are precisely ${n-1 \choose 2} - 1$ relations. It therefore suffices to demonstrate linear independence.

In the above table, every relation $\mathcal{R}(M)$ is paired with a boundary divisor $E(M)$. The following property is easily checked using Lemma~\ref{lem: boundary appears in RM iff row or column}: $E(M)$ appears in $\mathcal{R}(M)$, and $E(M)$ does not appear in any $\mathcal{R}(M^\prime)$ occurring to the right of $\mathcal{R}(M)$ in the table.

This implies that the set of relations is linearly independent: projecting onto the subspace based by the $E(M)$, the resulting square matrix is upper triangular with $1$s on the diagonal.	
\end{proof}

\footnotesize{
\bibliographystyle{alpha}
\bibliography{Bibliography.bib}\medskip

Patrick Kennedy-Hunt, University of Cambridge. Email: \href{mailto:pfk21@cam.ac.uk}{pfk21@cam.ac.uk}

Navid Nabijou, Queen Mary University of London. Email: \href{mailto:n.nabijou@qmul.ac.uk}{n.nabijou@qmul.ac.uk}

Qaasim Shafi, University of Birmingham. Email: \href{mailto:m.q.shafi@bham.ac.uk}{m.q.shafi@bham.ac.uk}

Wanlong Zheng, University of Cambridge. Email: \href{mailto:wz302@cam.ac.uk}{wz302@cam.ac.uk}
}
\end{document}